\documentclass[a4paper,10pt]{amsart}

\usepackage[utf8]{inputenc}
\usepackage[usenames, dvipsnames]{color}
\newtheorem{theorem}{Theorem}
\newtheorem{lemma}{Lemma}

\newcommand{\C}{\mathbb C}
\newcommand{\R}{\mathbb R}
\newcommand{\N}{\mathbb N}
\newcommand{\Z}{\mathbb Z}

\title[Universal entire functions that define order isomorphisms]{Universal entire functions that define order isomorphisms of countable real sets}

\author{P. M. Gauthier}

\address{Département de mathématiques et de statistique, Université de Montréal,
CP-6128 Centreville, Montréal,  H3C3J7, CANADA}
\email{gauthier@dms.umontreal.ca}

\begin{document}

\begin{abstract}
In 1895, Cantor showed that between every two countable dense real sets, there is an order isomorphism. In fact, there is always such an order isomorphism, that is the restriction of a universal entire function. 
\end{abstract}


\thanks{Research supported by NSERC (Canada) grant RGPIN-2016-04107}

\maketitle

\section{Introduction}

In 1895, Cantor proved that every two countable dense sets of reals are order isomorphic. The same year, Stäckel \cite{S1895} showed that, if $A$ is a countable and $B$ a  dense subset of $\C,$   then there exists  a non-constant entire function that maps $A$ {\em into} $B.$ He also claimed the corresponding result, if $A$ is a countable real set and $B$ is a dense real set. 

The following striking result was published by Franklin \cite{F1925}
in 1925.
 
\medskip
\noindent
{\bf Theorem F.} {\it Let $A$ and $B$ be countable dense subsets of $\R.$  Then, there exists an analytic function $f:\R\rightarrow \R,$ that restricts to an order isomorphism of $A$ onto $B.$} 

\medskip
Unfortunately, the proof invoked the statement that the uniform limit of analytic functions is analytic, which is false, as one can see, for example, from the Weierstrass approximation theorem. Fortunately, Franklin's theorem follows from a more general result of Burke \cite{B-TAMS2009}. 

In the present paper, we present the following two extensions of Franklin's Theorem.

\begin{theorem}\label{1}
Let $A$ and $B$ be countable dense subsets of $\R.$  Then, there exists an entire function $f,$ having finite order of growth, such that $f(\R)=\R,$ and  $f$ restricts to an order isomorphism of $A$ onto $B.$ Moreover,  $f^\prime(x)>0,$ for $x\in\R,$ so the mapping $f:\R\rightarrow\R$ is bianalytic. 
\end{theorem}

In proving Theorem \ref{1} we shall also show that $f$ can be further required to map a preassigned point $a\in A$ to a preassigned point $b\in B.$ 
Theorem \ref{1} is only a small improvement of Franklin's theorem, but we present the proof, firstly because of the error in Franklin's proof, but mainly because the proof we present can be adapted to prove the following. 

\begin{theorem}\label{2}
Let $A$ and $B$ be countable dense subsets of $\R.$   Then, there exists a universal entire function $f,$ such that: $f(\R)=\R,$  $f$ restricts to an order isomorphism of $A$ onto $B;$ and $f^\prime(x)>0,$ for $x\in\R.$  
\end{theorem}

Here, by a universal entire function, we mean an entire function which has the remarkable property that its translates are dense in the space of all entire functions. The existence of a universal entire function was proved in 1929 by George Birkhoff \cite{Bi1929}.

In higher dimensions, the following result, was proved by Morayne \cite{Mo1987} in 1987 for $\R^n$ and $\C^n$ and  by Rosay and Rudin \cite{RR1988} in 1988, with a different  proof, for $\C^n.$ 

\medskip
\noindent
{\bf Theorem MRR.} {\it Let $A$ and $B$ be countable dense subsets of $\C^n$ (respectively $\R^n$) $n>1.$ Then, there is a measure preserving  biholomorphic mapping of $\C^n$ (respectively bianalytic mapping of $\R^n$) that maps $A$ onto  $B.$}

\medskip
This theorem appears to be stronger than Franklin's Theorem, however the proof of Theorem MRR fails for $n=1.$ Moreover, for $n=1,$ all measure preserving automorphisms are of the form $z\mapsto  az+b, \, |a|=1,$ so the only automorphic images of a set $A$ are the sets $aA+b.$  

In 1957, Erdös \cite{E1957} asked whether, given countable dense subsets $A$ and $B$ of $\C,$ there exists an entire function $f$ that maps $A$ onto $B$ (compare Stäckel [{\em op. cit.}]). In 1967, Maurer \cite{Ma1967} gave an affirmative answer. In this context, there are the following two interesting results 
of Barth and Schneider \cite{BS1970} \cite{BS1972}, the second of which improves the result of Maurer.

\medskip
\noindent
{\bf Theorem BS1.} {\it Let A and B be countable dense subsets of $\R$. Then,
there exists an entire transcendental function $f$ such that $f(z) \in B$ if and only if $z\in A.$}

\medskip
\noindent
{\bf Theorem BS2.} {\it Let $A$ and $B$ be countable dense subsets of $\C.$ Then, there exists an entire function $f,$ such that $f(z)\in B$ if and only if $z\in A.$}

\medskip

Although the next result is not directly on the topic of the present paper, we consider it worth mentioning, perhaps as a distant cousin. 

\medskip
\noindent
{\bf Theorem P \cite{P2013}.}   {\it Let $A$ and $B$ be countable dense subsets of the Hilbert cube $H=[0,1]^\N.$  Then, for every $\epsilon>0,$ there is a measure preserving  homeomorphism $f$ of $H$ that maps $A$ onto  $B,$ and $\rho(f,id)<\epsilon,$ where $\rho$ is a distance on the set of continuous  mappings $H\rightarrow H$ and id is the identity mapping.}
\medskip

In recent years, Maxim Burke has obtained deep results of a nature similar to ours and Franklin's  see (\cite{B-JMAA2009}, \cite{B-TAMS2009} and \cite{B2017}).  
I thank the referees for helpful comments. 


\section{Proof of Theorem \ref{1}}

\begin{proof}
The desired  function $f$ will be of the form 
$$
	f(z) = \lim_{n\rightarrow\infty} f_n(z) = 
			z + \sum_{j=1}^\infty \lambda_j h_j(z),
$$
where $f_n(z)=z+\sum_{j=1}^n\lambda_j h_j(z)$ The purpose of the $h_j$'s, which we momentarily define, is to recursively make adjustments to obtain more of the desired mapping properties, without losing those properties which we have previously assured. 

Now, given a sequence $\{\epsilon_n\}$ of positive numbers, whose sum is less than $1,$   we shall construct an enumeration $(\alpha_n)$ of $A,$ an enumeration $(\beta_n)$ of $B$ and the real sequence $(\lambda_n)$ such that, taking the functions $h_j$ of the following form:
$$
	h_1=1; \quad  \mbox{and} \quad h_j(z) = e^{-z^2} \prod_{k=1}^{j-1} (z-\alpha_k),  \quad \mbox{for} \quad j=2,3,\ldots,
$$
we shall have that, for  $n=1,2,\ldots,$
\begin{equation}\label{interpolation}
	f_n(\alpha_j)=\beta_j, \quad j=1,\ldots,n, 
\end{equation}
\begin{equation}\label{disque} 
	\lambda_1 = \beta_1-\alpha_1 \quad \mbox{and} \quad
	|\lambda_n h_n(z)| < \epsilon_n, \quad \mbox{if} \quad |z|\le n, \quad n>1;
\end{equation}
\begin{equation}\label{droite}
	|\lambda_n h_n^\prime(x)| < \epsilon_n, \quad \mbox{if} \quad x\in\mathbb R;
\end{equation}
\begin{equation}
	f_n(\mathbb R)\subset \mathbb R.
\end{equation}
From the second condition, $f$ will be an entire function and the third will allow us to differentiate this series term by term on  $\mathbb R.$  
On $\mathbb R$   we have:
$$  
 f'(x) = 1 +  \sum_{j=1}^\infty \lambda_j h'_j(x) >  1  -  \sum_{j=1}^\infty \epsilon_j > 0.
$$
Hence, $f:\mathbb R\rightarrow \mathbb R$ is strictly increasing and consequently injective. Moreover, since $f(x)\rightarrow\pm\infty,$ as $x\rightarrow\pm\infty,$ the function $f:\R\rightarrow\R$ is surjective and, consequently bianalytic. 

Now, we shall show that, if the $\lambda_n$ are chosen even smaller, the function $f$ will be of finite order. Indeed, at stage $n,$ the $\alpha_1, \ldots, \alpha_{n-1}$ having been chosen, we may choose $\lambda_n\in\mathbb R$ so small that 
\begin{equation}\label{order}
	|\lambda_n|\prod_{k=1}^{n-1}|z-\alpha_k|<e^{|z|}/2^n, \quad \mbox{for all} \quad z\in \C.
\end{equation}
Thus, 
$$
	|f(z)| \le |z| + |\lambda_1| + e^{|z|^2}\sum_{n=2}^\infty e^{|z|}/2^n \le 
		|z|+|\lambda_1| + e^{|z|^3}, \quad \mbox{for all} \quad z\in \C,
$$
so $f$ is of finite order. 

Choose $a_1\in A$ and $b_1\in B.$
Now, we shall  choose the sequences  $\{\alpha_n\}, \{\beta_n\}$ and $\{\lambda_n\}.$ 
First, we choose enumerations  $\{a_n\}$ and $\{b_n\}$ of $A$ and $B.$  
The sequences  $\{\alpha_n\}$ and $\{\beta_n\}$ will be rearrangements of   $\{a_n\}$ and $\{b_n\}$ chosen recursively.
Set:
$$\alpha_1 = a_1 \quad \beta_1 \quad \mbox{and} \quad  \lambda_1 = \beta_1-\alpha_1.$$ 
We have defined $\alpha_1, \lambda_1, \beta_1$ and hence also $h_1,$  $f_1$ and  $h_2.$ Note that 
$$f_1(a_1) = f_1(\alpha_1) = \beta_1 = b_1.$$   
Let  $\beta_2$ be the first $b_j$ not equal to $\beta_1 = b_1.$ In fact $\beta_2=b_2.$ 

Suppose we have chosen distinct members  
$\alpha_1,\ldots,\alpha_{2n-1}$ of the sequence $\{a_i\},$ such that,  for each $k=1,\cdots,n,$ $\alpha_{2k-1}$ is the first $a_i$ not previously chosen;
distinct $\beta_1,\ldots,\beta_{2n}$ from the sequence $\{b_j\},$ such that, for each $k=1,\ldots,n,$ $\beta_{2k}$ is the first $b_j$ not previously chosen; 
real numbers $\lambda_1,\ldots,\lambda_{2n-1},$ such that conditions (\ref{disque}), 
(\ref{droite}) and (\ref{order}) are satisfied.
We shall now choose 
$\alpha_{2n}, \lambda_{2n}, \alpha_{2n+1},  \lambda_{2n+1}, \beta_{2n+1}$ and $\beta_{2(n+1)}.$ 

Choose $\eta>0$ such that (\ref{disque}), (\ref{droite}) and (\ref{order}) hold for $|\lambda|<\eta.$ Now, it follows from  (\ref{droite}) that  $f_{2n-1}^\prime(x)\ge 1-\sum_{j=1}^\infty\epsilon_j>0$ for all $x\in\mathbb R.$ Thus the function $f_{2n-1}$ is surjective on $\mathbb R.$ Choose a number $x_n$ such that $f_{2n-1}(x_n)=\beta_{2n}.$ Note that $h_{2n}(x_n)\not=0.$
To see this, suppose, to obtain a contradiction, that  $h_{2n}(x_n)=0.$ Then, $x_n=\alpha_j,$ for some $j=1,\ldots,2n-1$ and $f_{2n-1}(\alpha_j)=\beta_{2n}.$ However, $f_{2n-1}(\alpha_j)=\beta_j.$ Thus, $\beta_{2n}=\beta_j.$ But this contradicts the choice of $\beta_{2n}$ as being distinct from $\beta_j,$ for $j<2n.$ Since $h_{2n}(x_n)\not=0,$ the function
$$
	\lambda(x) = \frac{\beta_{2n}-f_{2n-1}(x)}{h_{2n}(x)}
$$
is defined and continuous in a neighbourhood $I$ of $x_n,$ with $\lambda(x_n)=0;$ by choosing $I$ smaller we can also have that $|\lambda(x)|<\eta$ for $x\in I.$ By the density of $A$ there is some $\alpha_{2n}\in A\setminus \{\alpha_1,\ldots,\alpha_{2n-1}\}$ in $I.$ Write $\lambda_{2n}=\lambda(\alpha_{2n}).$ Then we have that $|\lambda_{2n}|<\eta$ and $f_{2n-1}(\alpha_{2n})+\lambda_{2n}h_{2n}(\alpha_{2n})=\beta_{2n}.$ This implies  (\ref{interpolation}), (\ref{disque}), (\ref{droite}) and (\ref{order}) for $2n.$

The choice of $\alpha_{2n+1}$ is easy.  
We choose the first of the $a_j $  different from 
$\alpha_1, \ldots, \alpha_{2n} $
and call it $\alpha_{2n+1}. $ 

Since $h_{2n+1}(\alpha_{2n+1})\not=0,$ the linear function 
$$\beta(\lambda) =  f_{2n}(\alpha_{2n+1})+\lambda h_{2n+1}(\alpha_{2n+1})$$ is non-constant. Hence $J_n=\{\beta(\lambda):|\lambda|<\epsilon\}$ is a non-empty open interval and, since $B$ is dense, we may choose an element of $(B\cap J_n)\setminus \{\beta_1,\cdots,\beta_{2n}\},$ which we call $\beta_{2n+1}.$ The element $\beta_{2n+1}$ by definition has the form $\beta_{2n+1}=\beta(\lambda)$ for a  certain $\lambda,$ with $|\lambda|<\epsilon.$ We denote this $\lambda$ by $\lambda_{2n+1}.$
If $\epsilon$ is sufficiently small, then $\lambda_{2n+1}$ satisfies   (\ref{disque}), (\ref{droite}) and (\ref{order}). Condition  (\ref{interpolation}) is satisfied by the choice we have just made for
$\beta_{2n+1}$ and $\lambda_{2n+1}.$ 

For $\beta_{2(n+1)}$ we choose the first of the $b_j$'s different from 
$\beta_1, \ldots,\beta_{2n+1}.$ 

The construction of the sequences  $(\alpha_n), (\lambda_n),$ and $(\beta_n),$ 
and hence also of the the entire function is complete.  This concludes the proof of Theorem \ref{1}. 

\end{proof}


\section{Approximation by Entire Functions}

For a  set $S\subset\C$ we denote by $S^0$ the interior of $S.$ We 
say that a function $f:S\rightarrow\C$ is holomorphic on $S$ if there is an open neighbourhood $U$ of $S$ and a holomorphic function $F$ on $U,$ such that $F=f$ on $S.$  We denote by   $H(S)$ the class of functions holomorphic on $S$ and by $A(S)$ the class of functions continuous on $S$ and holomorphic on $S^0.$  The extended complex plane is denoted by $\overline \C.$ Let $E\subset \C$ be symmetric with respect to the real axis. We shall say that a  function $f:E\rightarrow \C$ defined on such a  set $E$ is symmetric with respect to the real axis, if $f(\overline z) = \overline{f(z)},$ for $z\in E.$ 

A compact set $K\subset\C$ is a {\em Mergelyan set} if every $f\in A(K)$ can be uniformly approximated by polynomials. 

\medskip
\noindent
{\bf Mergelyan Theorem.} {\em A compact set $K\subset \C$ is a Mergelyan set if and only if $\C\setminus K$ is connected. Moreover, if $K$ is symmetric with respect  to the real axis, $f\in A(K)$ and $f(\overline z)=\overline{f(z)}, z\in K,$  the approximating polynomials can be taken with real coefficients.}
\medskip

\begin{proof}
To verify the last statement, which is not part of the original Mergelyan Theorem, suppose $K$ is symmetric with respect to the real axis, $f\in A(K)$  and $f(\overline z)=\overline{f(z)}, z\in K.$ Let $p_n, \, n=1,2,\ldots,$ be a sequence of polynomials that converges uniformly to $f$ and set $q_n(z)= (p_n(z) + \overline{p_n(\overline z)})/2.$ Then, the sequence of polynomials $q_n$ also converges uniformly to $f$ and moreover have real coefficients.  
\end{proof}

For a topological vector space $X,$ we denote by $X^*$ the continuous dual space. 
The following Walsh-type lemma on simultaneous approximation and interpolation is due to Frank Deutsch.

\medskip
\noindent
{\bf Walsh Lemma \cite{D1966}.} {\em Let $X$ be a locally convex topological complex vector space and $Y$ a dense subspace. Then, if $x\in X,$ $U$ is a neighbourhood of $0$ and $L_1,\ldots,L_n\in X^*,$ there is a $y\in Y,$ such that} 
$$
	y\in x+U \quad \mbox{and} \quad  L_j(y)=L_j(x), \quad \mbox{for}  \quad j=1,\ldots,n.
$$

Let $E$ be a closed set that is starlike with respect to the origin. For $f:E\rightarrow\C,$ we shall write $f\in A^1(E),$ if $f\in H(E^0),  \, f$ has a radial derivative at each point of $E,$ which by  abuse of notation, we shall also denote as $f^\prime$ and $f^\prime$ is continuous on $E.$ We note that there are many other meanings assigned to the notation $A^1(E)$ in the literature.

\begin{lemma}\label{KE}

Let $K\subset\mathbb C$ be compact  and starlike with respect to the origin; let $E$ be a Mergelyan set disjoint from $K$ and set $Q=K\cup E.$ Suppose $f\in A(Q); \, f|_K\in A^1(K)$ and $z_1,\ldots,z_n$ are distinct points in $K.$ 
Then, for every $\epsilon>0,$ there is a polynomial $p$ such that
$$
	|p-f|_Q<\epsilon, \,  |p^\prime-f^\prime|_K<\epsilon, 
		\quad p(z_j)=f(z_j), \,\, p^\prime(z_j)=f^\prime(z_j), \quad j=1,\ldots, n. 
$$
If $Q$ is symmetric with respect to the real axis, the $z_j$ are real numbers and $f(\overline z)=\overline{f(z)}, z\in Q,$ we may take  $p$  with real coefficients.

\end{lemma}

\begin{proof}
Without loss of generality, we may assume that $f(0)=0.$ Choose a number $r>1$ such that $r>|z|,$ for all $z\in K.$ 
By Mergelyan's Theorem, there is a polynomial $q,$ for which $|q(z)-f^\prime(z)|<\epsilon/r<\epsilon,$ for all $z\in K.$ Consider the polynomial 
$$
	p(z) = \int_0^zq(\zeta)d\zeta. 
$$
Clearly $|p^\prime-f^\prime|<\epsilon.$ Moreover, denoting by $[0,z]$ the segment from $0$ to $z,$ we have
$$
	|p(z)-f(z)| = \int_{[0,z]}\left|q(\zeta)-f^\prime(\zeta)\right||d\zeta|<\epsilon.
$$
Now, if $f(\overline z)=\overline{f(z)}, z\in K,$ then replacing $p(z)$ by $(p(z)+\overline p(\overline z))/2,$ we obtain a $p$ with real coefficients

Consider $X=A^1(K)\cap A(E),$ endowed with its canonical norm. Then, by Mergelyan and the preceding paragraph, the polynomials are dense in $X.$ In addition point evaluation and point derivation at points of $K$ are continuous linear functionals. Then the Walsh Lemma implies the result. 
\end{proof}

A closed subset $E\subset \C$ is called an {\em Arakelian set}, if for every function $f\in A(E)$ and every positive {\em number} $\epsilon,$ there is an entire function $g,$ such that $|f(z)-g(z)|<\epsilon,$ for all $z\in E.$  By Arakelian's Theorem,  $E$ is an Arakelian set if and only if $\overline{\mathbb C}\setminus E$ is connected and locally connected (see \cite{G1987}).

A closed subset $E\subset \C$ is called a {\em Carleman set}, if for every function $f\in A(E)$ and every positive continuous {\em function} $\epsilon$ on $E,$ there is an entire function $g,$ such that $|f(z)-g(z)|<\epsilon(z),$ for all $z\in E.$ Obviously, a Carleman set is an Arakelian set. In the previous millenium, I found a necessary condition in order for an Arakelian set to be a Carleman set - namely, that for every compact set $K\subset\C,$ there is a compact set $Q\subset\C$ such that every component of $E^0$ that meets $K$ is contained in $Q.$ This condition is sometimes described by saying that ``the interior of $E$ has no long islands". Nersesyan showed that this condition is also sufficient. Thus, an Arakelian set is a Carleman set if and only if its interior has no long islands. Thus, every Arakelian set having no interior is a Carleman set. In particular, the real line $\R$ is an Arakelian set and hence a Carleman set. In defining a Carleman set, by the Tietze extension theorem \cite{D1978} on closed sets, it makes no difference whether we consider the continuous function  $\epsilon$ to be defined on the closed set  $E$ or on all of $\C.$ 
For references and more information on Arakelian and Carleman sets, see \cite{G1987}. 

Since $\R$ is a Carleman set, every continuous function $f(x)$ on $\R$ can be approximated by an entire function $g(x)$ within $\epsilon(x),$ with $\epsilon(x)\searrow 0$ arbitrarily fast, as $x\rightarrow\infty.$ The following theorem of Hoischen  asserts that, if $f$ is not only continuous, but also continuously differentiable, there is an entire function $g,$  such that $g(x)$ approximates $f(x),$  $g^\prime(x)$ approximates $f^\prime(x)$ and moreover, $g$ and $g^\prime$ interpolate $f$ and $f^\prime$ respectively on a discrete subset of $\R.$ 

\medskip
\noindent
{\bf Hoischen`s Theorem  \cite{H1975}.}
{\em Let $f\in C^1(\R),$ $\epsilon$ be a positive continuous function on $\R$ and $X$ be a discrete subset of $\R.$ Then, there is an entire function $g,$ such that 
$$
	\max\{|g-f|,|g^\prime-f^\prime|\}<\epsilon, \quad \mbox{and} \quad
\quad g(x)=f(x), \, g^\prime(x)=f^\prime(x), \,\, \forall x\in X.
$$
If $f$ is real-valued, then we may take $g$ to also be real-valued on $\R.$ 
}

\medskip

The last sentence of the theorem is not stated in \cite{H1975}, but follows in the usual manner by replacing the function $g(z)$ by the function $(g(z)+\overline{g(\overline z})/2.$

 A {\em chaplet} is a closed set that is the union of an infinite family of disjoint closed discs that is  locally finite (i.e. each compact set meets at most finitely many members of the family). Henceforth, $E$ will denote a chaplet that is disjoint from the real axis and is symmetric with respect to the real axis. Thus, $E$ is the union of an infinite but locally finite family of disjoint closed discs $E^+_n$ in the open upper half-plane and their reflections $E^-_n$ in the open lower half-plane. We shall suppose that the radii of the $E^\pm_n$ tend to infinity. We shall also suppose that these discs are ordered and separated in the following sense.  There is a sequence $r_n>0, r_n\nearrow \infty,$ such that $E^\pm_n$ is contained in the annulus $r_n<|z|<r_{n+1},$ for each $n.$ A chapelet $E$ having all of these properties will be called a {\em special chaplet.}

\begin{lemma}\label{RE}

Let $E$ be a special chaplet and set $F=\R\cup E.$  Suppose  $\{x_k\}, k\in\Z,$ is a strictly increasing sequence of real numbers tending to $\infty,$ as $n\rightarrow\infty$ and   $\epsilon$ is a positive continuous function  on $\C,$ Then, for every  function $f\in A(F)$ such that $ f|_\R\in C^1(\R)$ and $f(\overline z) = \overline{f(z)}$ for all $z\in F,$ there exists an entire function $g,$ such that $g(\overline z) = \overline{g(z)}$ for all $z\in \C$ and
$$
	|g(z)-f(z)|<\epsilon(z), \, \forall z\in F; \quad 
	|g^\prime(z)-f^\prime(z)|<\epsilon(z), \, \forall z\in\R;
$$
$$
		\quad g(x_k)=f(x_k), \,\, g^\prime(x_k)=f^\prime(x_k), \quad k\in\Z. 
$$
 
\end{lemma}

\begin{proof}  Note that $X=\{x_k:k\in\Z\}$ is a discrete subset of $\R.$ 
By Hoischen's Theorem, there is an entire function $\varphi$ such that  $\varphi(\overline z) = \overline{\varphi(z)},$ 
$$
	\max\{|f(x)-\varphi(x)|,|f^\prime(x)-\varphi^\prime(x)|\}<\epsilon(x)/2, \quad \mbox{for all} \quad x\in\R
$$
and
$$
		f (x_k)=\varphi (x_k), \quad k\in\Z.
$$

For $n=1,2,\ldots,$ set $\overline{D_n}=\{z: |z|\le r_n\},$ where $\{r_n\}$ is the sequence of separating radii for the chaplet $E,$ and set 
$$
	K_n = [-r_{n+1},-r_n]\cup \overline{D_n} \cup [r_n,r_{n+1}].
$$

Without loss of generality, we may assume that $\epsilon(z)=\epsilon(|z|)$ and that $\epsilon(r)$ is strictly decreasing on $[0,+\infty).$ 
Let $\epsilon_1,\epsilon_2,\ldots,$ be a strictly decreasing sequence of positive numbers, such that
$$
	\epsilon_n<\min_{z\in K_n\cup E_n}\epsilon(z)=\epsilon(r_{n+1}) \quad \mbox{and} \quad
	\sum_{k=n+1}^\infty 2\epsilon_k < \epsilon_n, \quad n=1,2,\ldots. 
$$

The compact sets $K_n$ are starlike with respect to the origin and symmetric with respect to the real axis. The union $E_n$ of the two closed discs $E_n^\pm$ is a Mergelyan set disjoint from $K_n.$ Each compact set $Q_n=K_n\cup E_n$ satisfies the hypotheses of Lemma \ref{KE} and we shall recursively define corresponding functions $f_n.$ 

We define $f_1\in A^1(Q_1),$ by setting  
$$
f_1(z) = \left\{
\begin{array}{ll}
	\varphi(z),	&	z\in K_1\\
	f(z),		&	z\in E_1.
\end{array}
\right.
$$
By Lemma \ref{KE}, there is a polynomial $p_1,$ with real coefficients, such that
$$
	|p_1-f_1|_{Q_1}<\epsilon_2, \quad  |p_1^\prime-f_1^\prime|_{K_1}<\epsilon_2, 
$$
and, for $X_1= \{x_k: x_k\in K_1\},$ 
$$
p_1(x)=\varphi(x), \, x\in X_1, \quad p_1(\pm r_2)=\varphi(\pm r_2), \quad p_1^\prime(\pm r_2)=\varphi^\prime(\pm r_2). 
$$

Set  $p_0=\varphi$ and suppose, for $n\ge 1$ and $k=1,\ldots, n-1,$ we already have polynomials $p_k,$ with real coefficients, such that, for
$$
f_k(z) = \left\{
\begin{array}{ll}
	p_{k-1}(z),	&	z\in \overline D_k\\
	\varphi(z),		&	z\in[-r_{k+1},-r_k]\cup[r_k,r_{k+1}]\\
	f(z),			&	z\in E_k,
\end{array}
\right.
$$
we have 
$$
	|p_k-f_k|_{Q_k}<\epsilon_{k+1}, \quad  
		|p_k^\prime-f_k^\prime|_{K_k}<\epsilon_{k+1}, 
$$
and, for $X_k= \{x_j: x_j\in K_k\},$
$$
p_k(x)=\varphi(x), \, x\in X_k, \quad p_k(\pm r_{k+1})=\varphi(\pm r_{k+1}), \quad p_k^\prime(\pm r_{k+1})=\varphi^\prime(\pm r_{k+1}). 
$$

We define $f_n\in A^1(Q_n),$ by setting
$$
f_n(z) = \left\{
\begin{array}{ll}
	p_{n-1}(z),	&	z\in \overline D_n\\
	\varphi(z),		&	z\in[-r_{n+1},-r_n]\cup[r_n,r_{n+1}]\\
	f(z),			&	z\in E_n.
\end{array}
\right.
$$
By Lemma \ref{KE}, there is a polynomial $p_n,$ with real coefficients, such that
$$
	|p_n-f_n|_{Q_n}<\epsilon_{n+1}, \quad  |p_n^\prime-f_n^\prime|_{K_n}<\epsilon_{n+1}, 
$$
and, for $X_n = \{x_k: x_k\in K_n\},$ 
$$
p_n(x)=\varphi(x), \, x\in X_n, \quad p_n(\pm r_{n+1})=\varphi(\pm r_{n+1}), \quad p_1^\prime(\pm r_{n+1})=\varphi^\prime(\pm r_{n+1}). 
$$
By induction, the polynomials $p_n$ are now defined for all $n=1,2, \ldots.$

Fix positive integers $k<m<n.$ On $\overline D_k,$ we have
$$
	|p_n(z)-p_m(z)|\le \sum_{j=m}^{n-1}|p_{j+1}(z)-p_j(z)|<\sum_{j=m}^{n-1}\epsilon_{j+1}
$$
and, since $\sum\epsilon_j$ is convergent, the sequence $p_n$ is uniformly Cauchy on each $\overline D_k$ and hence  converges uniformly on compact subsets to an entire function $g.$ Of course, we also have that $p_n^\prime\rightarrow g^\prime$ uniformly on compact subsets.  Since all of the $p_n$ have real coefficients, $g(\overline z)=\overline{g(z)}.$

Fix $z\in E.$ Then, $z\in E_m,$ for some $m$ and $f(z)=f_m(z).$ Choose $n>m$ such that $|g(z)-p_n(z)|<\epsilon_{m+1}.$ Then, 
$$
	|g(z)-f(z)|\le |g(z)-p_n(z)| + |p_n(z)-f_m(z)| \le \epsilon_{m+1} +\sum_{m+1}^{n+1}\epsilon_k<\epsilon_m<\epsilon(z).
$$

There remains to show that $g$ has the desired approximation and interpolation properties on $\R.$ To show that
$$
	|g-f|_\R<\epsilon, \quad  |g^\prime-f^\prime|_\R<\epsilon, 
$$ 
it suffices, by the triangle inequality, to show that
$$
	|g-\varphi|_\R<\epsilon/2, \quad  |g^\prime-\varphi^\prime|_\R<\epsilon/2, 
$$  
Fix $x\in\R.$ Let $m=m_x$ be the minimum $m,$ such that $x\in K_m.$ For $n\ge m,$ 
$$
	|p_n(x)-\varphi(x)| = |p_n(x)-f_m(x)| \le  |p_m(x)-f_m(x)| + \
		\sum_{k=m}^{n-1}|p_{k+1}(x)-p_k(x)|\le
$$
$$
	\epsilon_{m+1} + \sum_{k=m}^{n-1}\epsilon_{k+2} <
		 \sum_{k=m+1}^\infty\epsilon_k<\epsilon_m/2<\epsilon(x)/2.
$$
Thus, $|g(x)-\varphi(x)|<\epsilon(x)/2,$ for all $x\in\R.$ From the triangle inequality, $|g(x)-f(x)|<\epsilon(x),$ for all $x\in\R.$ The proof that $|g^\prime(x)-f^\prime(x)|<\epsilon(x),$ for all $x\in\R$ is completely analogous.

Now, fix $k\in \Z$ and let $m$ be the first $m$ for which $x_k\in K_m.$ Then, for all $n\ge m,$ we have $p_n(x_k)=\varphi(x_k)=f(x_k)$ and $p_n^\prime(x_k)=\varphi^\prime(x_k)=f^\prime(x_k).$
Since $p_n\rightarrow g$ and  $p_n^\prime\rightarrow g^\prime,$ we obviously have $g(x_k)=f(x_k)$ and $g^\prime(x_k)=f^\prime(x_k).$

\end{proof}

\begin{lemma}\label{Phi}
Let $E$ be a special chaplet and let  $\epsilon$ be a positive continuous function on $\C.$ Then, for every  function $f\in A(E)$ such that $f(\overline z) = \overline{f(z)},$ there exists an entire function $\Phi,$ such that $|f-\Phi|<\epsilon$ on $E;$  $\Phi$ maps $\R$ bianalytically onto $\R;$ and $\Phi^\prime>0$ on $\R.$  
\end{lemma}

\begin{proof}
Set $\psi(x)=x,$ for $x\in\R.$  We may assume that $\epsilon(x)<1/2 = \psi^\prime(x)/2.$

Set $F=\R\cup E$ and extend $f,$  by setting $f=\psi$ on $\R.$ Then, $f\in A(F)$ and $f|_\R\in C^1(\R).$ By Lemma \ref{RE} there exists an entire function $\Phi,$ such that $\Phi(\overline z) = \overline{\Phi(z)}$ and
$$
	|\Phi-f|_F<\epsilon, \,  |\Phi^\prime-f^\prime|_\R<\epsilon. 
$$
Now,
$$
	\Phi^\prime(x) = f^\prime(x)-(f^\prime(x)-\Phi^\prime(x)) >
		 \psi^\prime(x)-\psi^\prime(x)/2 > 0.
$$
Hence  $\Phi^\prime>0$ on $\R.$ 
Thus, $\Phi:\R\rightarrow \R$ is injective. 
Since  the restriction $\Phi:\R\rightarrow\R$ is neither lower nor upper bounded and continuous, it follows that it is surjective. We have verified that $\Phi:\R\rightarrow\R$ is a bijection and, since $\Phi^\prime>0,$ it is bianalytic. 

\end{proof}

\begin{lemma}\label{H}
Let $E$ be a special chaplet and   $\epsilon$ be a positive continuous function on $\C.$ Then, there exists an entire function $H,$ such that $|H|<\epsilon$ on $E;$  $H(\overline z)=\overline{H(z)};$ 
$$
	|H(x)-1| < \epsilon(x) \quad \mbox{and} \quad
	 |H^\prime(x)|<\epsilon(x), \quad \mbox{for all} \quad x\in\R.
$$ 
\end{lemma}

\begin{proof}
In Lemma \ref{RE} we set $f$ equal to $0$ on $E$ and $1$ on $\R.$ 
\end{proof}


\smallskip
\section{Proof of Theorem 2}

\begin{proof}
The proof follows the steps in the proof of Theorem 1. However, in place of
$$
	f(z) = 
z + \sum_{j=1}^\infty \lambda_j h_j(z),
$$
$$
	h_1=1; \quad  \mbox{and} \quad h_j(z) = e^{-z^2} \prod_{k=1}^{j-1} (z-\alpha_k),  \quad \mbox{for} \quad j=2,3,\ldots,
$$
our desired function $f$ will be of the  form 
\begin{equation}\label{universal}
	f(z) = \lim_{n\rightarrow\infty} f_n(z) = 
			\Phi(z) + H(z)\sum_{j=1}^\infty \lambda_jh_j(z),
\end{equation}
$$
	h_1=1; \quad  \mbox{and} \quad h_j(z) = e^{-\Phi^2(z)} \prod_{k=1}^{j-1} \left(\Phi(z)-\Phi(\alpha_k)\right),  \quad \mbox{for} \quad j=2,3,\ldots,
$$
The functions $\Phi$ and $H$ will come respectively from Lemmas \ref{Phi} and \ref{H}. The purpose of the function $\Phi$ will be to assure that $f$ has all the desired properties except possibly the main one (bijection of $A$ onto $B$). The purpose of the $h_j$'s is to recursively make adjustments to obtain the main condition. The purpose of the function $H,$ which is small on $E$ and close to $1$ on $\R$ is to combine the desired behaviour on $E$ with the desired behaviour on $\R.$ The function $H$ will be sufficiently small on $E$ that adding the series does not destroy the universality achieved  by $\Phi.$

Let $E^\pm_n$ be the closed discs making up a special chaplet $E,$ and denote by $a_n^\pm$  and $\rho_n$ respectively the center and radius of $E^\pm_n.$ Denoting by $\overline D_n$ the closed disc centered at the origin of radius $\rho_n,$ we have $E^\pm_n=\overline D_n+a_n^\pm.$ 

{\em Claim $\Phi.$
There is an entire function $\Phi,$ such that $\Phi$ maps $\R$ bianalytically onto $\R,$  $\Phi^\prime>0$ on $\R$  and the sequence $\Phi(\cdot+a_n^+)$ of translates of $\Phi$ is dense in the space of entire functions. Thus, $\Phi$ is universal.}

To establish the claim, 
let $p_n, \, n=1,2,\ldots,$ be a  sequence of all the  polynomials whose coefficients have both real and imaginary parts rational. Since these polynomials are dense in the space of entire functions, an entire function will be universal, providing its translates approximate each $p_n.$ We shall assume that each polynomial of the sequence occurs infinitely often in the sequence. Define a function $\varphi\in H(E)$ by setting $\varphi(z)=p_n(z-a_n^+)$ on $E_n^+$ and $\varphi(z)=\overline{p_n(\overline{z-a_n^-})},$ on $E_n^-,$ for $n=1,2,\ldots.$ By Lemma \ref{Phi}, there exists an entire function $\Phi,$ such that  $\Phi$ maps $\R$ bianalytically onto $\R;$ $\Phi^\prime>0$ on $\R;$ $\Phi(\overline z)=\overline{\Phi(z)}$  and
\begin{equation}\label{f-varphi}
	|\Phi(z)-\varphi(z)|<1/n, \quad \mbox{for all} \quad z\in E_n^\pm, \quad n=1,2,\ldots.
\end{equation}
We claim that the sequence $\Phi(\cdot+a_n^+)$ is dense in the space of entire functions.  Indeed, fix an entire function $g,$ a compact set $K$ and an $\epsilon>0.$ There is a $p_m$ such that $|p_m-g|<\epsilon$ on $K.$ 
For all sufficiently large $n,$ we have $K\subset \overline D_n$ and $1/n<\epsilon.$ There is such an $n$ with  $p_n=p_m.$ Then $|\Phi(z)-p_n(z-a_n^+)|<1/n$ on $E_n^+,$ hence $|\Phi(z+a_n^+)-p_n(z)|<1/n$ on $\overline D_n.$ Therefore
$$
	|\Phi(z+a_n^+)-g(z)|_K\le|\Phi(z+a_n^+)-p_n(z)|_K+|p_n(z)-g(z)|_K<1/n+\epsilon<2\epsilon.
$$
Thus, the sequence $\Phi(\cdot+a_n^+)$ of translate of $\Phi$  is indeed dense in the space of entire functions. This establishes {\em Claim $\Phi$.}

{\em Claim $H.$ There is an entire function $H,$ such that  $H(\overline z)=\overline{H(z)},$
$$
0<H(x)<2,  \quad H^\prime(x) h_j(x)=O(\Phi^\prime(x)), \quad j\ge 1, 
		\quad\quad x\in \R
$$
and
$$
	H(z)h_j(z)<O(1/z), \quad j\ge 1, \quad z\in E.
$$}

To verify this claim, we begin by choosing a positive continuous function such that
$$
	\epsilon(z) < \min\{1, |h_1(z)z|^{-1},\ldots,|h_n(z)z|^{-1}\}, 
		\quad z\in E_n;
$$ 
$$
	\epsilon(x) < \min\{1,\Phi^\prime(x)h_1^{-1}(x),\ldots,\Phi^\prime(x)h_n^{-1}(x)\}, 
		\quad |x|\ge n-1, \quad n\ge 1.
$$
The positive continuous function $\epsilon$ is now defined on $\R\cup E$ and we may extend it to all of $\C$ by the Tietze extension theorem for closed sets \cite{D1978}. 

Now, let $H$ be an entire function associated to $\epsilon$ by Lemma \ref{H} and fix $j.$ Then $0<|H(x)-1|<\epsilon(x)<1,$ hence $0<H(x)<2.$  For $|x|\ge j, \, |H^\prime(x)h_j(x)|\le|\Phi^\prime(x)|,$ so $H^\prime h_j = O(\Phi^\prime).$ Now, suppose $z\in E.$ Then, $z\in E_n,$ for some $n=n(z).$ For all sufficiently large $z,$ we have $n(z)>j,$ so $|H(z)|<\epsilon(z)\le |h_j(z)z|^{-1}.$ Thus, $H(z)h_j(z)=O(1/z)$ on $E.$  We have verified that $H$ satisfies the Claim H.

Bearing in mind Claims $\Phi$ and $H$  for $\Phi$ and $H$ respectively, we can imitate the proof of Theorem \ref{1} to construct the appropriate sequences $\alpha_n, \beta_n$ and $\lambda_n,$ so that the function given by (\ref{universal}) has all of the properties  required for Theorem \ref{2}.

By choosing the $\lambda_n$'s sufficiently small, it is easy to ensure that the series converges uniformly on compacta, so that $f$ is an entire function. 

{\em Claim $f^\prime.$} We may  choose the  the $\lambda_n$'s so small that $f_n^\prime>0$ and  $f^\prime(x)>0,$ for $x\in\R.$ To see this, firstly we have 
$$
	f^\prime(x) \ge \Phi^\prime(x) - \sum_{j=1}^\infty\left(|H^\prime(x)\lambda_jh_j(x)|+
		|H(x)\lambda_jh_j^\prime(x)|\right).
$$
To establish the claim, it is sufficient to verify that, as $x\rightarrow \infty,$  
$$
	H^\prime h_j = O(\Phi^\prime) \quad \mbox{and} \quad Hh_j^\prime = O(\Phi^\prime).
$$
The first of these follows from Claim H. The second follows from the fact that $e^{-t^2}t^m=O(1)$ for every natural number $m.$ This establishes Claim $f^\prime.$

From Claim $f^\prime,$ we have that $f_n:\R\rightarrow \R$ and $f:\R\rightarrow \R$ are strictly increasing. We may  further choose the $\lambda_n$'s so small that the series is bounded on $\R.$ Since $\Phi$ is neither upper nor lower bounded, the same is true of $f_n$ and $f,$ and since $f_n$ and $f$ are continuous, it follows that $f_n:\R\rightarrow \R$ and $f:\R\rightarrow \R$ are surjective. We now have that $f_n:\R\rightarrow\R$ and $f:\R\rightarrow\R$ are bijections and since  $f_n^\prime(x)>0, f^\prime(x)>0, \, x\in \R,$ the mappings $f_n:\R\rightarrow \R$ and $f:\R\rightarrow \R$ are bianalytic.

The recursive choice of the sequences $\alpha_n, \beta_n$ and $\lambda_n$ is the same as in the proof of Theorem \ref{1}, but the difference between Theorem \ref{1} and Theorem \ref{2} justify our presenting this part of the proof again in this new context. 

First, we choose enumerations  $\{a_n\}$ and $\{b_n\}$ of $A$ and $B,$ where $a_1$ and $b_1$ are respectively preassigned values of $A$ and $B.$   The sequences  $\{\alpha_n\}$ and $\{\beta_n\}$ will be rearrangements of   $\{a_n\}$ and $\{b_n\}$ chosen recursively. Set 
$$
\alpha_1 = a_1, \quad \beta_1 = b_1 \quad \mbox{and} \quad \lambda_1=\frac{\beta_1-\Phi(\alpha_1)}{H(\alpha_1)}. 
$$ 
Thus, $f_1(\alpha_1)=\beta_1.$ Set $\beta_2=b_2.$ 

We proceed by induction. Suppose  we have respectively distinct
$$
\alpha_1,\ldots, \alpha_{2n-1};\quad \lambda_1, \ldots, \lambda_{2n-1}; 
\quad \beta_1, \ldots, \beta_{2n}
$$
$$
\alpha_{2k-1} =  (\mbox{first} \, a_i)\in A\setminus\{\alpha_j:j< 2k-1\}, \quad k=1,\ldots,n
$$
$$
\beta_{2k} =  (\mbox{first} \, b_i)\in B\setminus\{\beta_j:j< 2k\}, \quad k=1,\ldots,n
$$
$$
f_{2n-1}(\alpha_j)=\beta_j, \quad j=1,\ldots, 2n-1.
$$
To complete the induction, it suffices to choose  
$\alpha_{2n}, \lambda_{2n},
\beta_{2n+1}, \alpha_{2n+1}, \lambda_{2n+1}, 
\beta_{2(n+1)}$\\
with
$$
	f_{2n}(\alpha_{2n}) = \beta_{2n}, \quad\quad f_{2n+1}(\alpha_{2n+1}) = \beta_{2n+1}.
$$
Now, we explain how this can be done.

Firstly, we shall choose $\alpha_{2n}$ and $\lambda_{2n}$ so that $f_{2n}(\alpha_{2n})=\beta_{2n}.$ 
Choose $\eta>0$ such the (finitely many) smallness conditions  hold for $|\lambda_{2n}|<\eta.$ We have already noted that   $f_{2n-1}$ is surjective on $\mathbb R.$ Choose a number $x_n$ such that $f_{2n-1}(x_n)=\beta_{2n}.$ Note that $h_{2n}(x_n)\not=0.$
To see this, suppose, to obtain a contradiction, that  $h_{2n}(x_n)=0.$ Then, $x_n=\alpha_j,$ for some $j=1,\ldots,2n-1$ and $f_{2n-1}(\alpha_j)=\beta_{2n}.$ However, $f_{2n-1}(\alpha_j)=\beta_j.$ Thus, $\beta_{2n}=\beta_j.$ But this contradicts the choice of $\beta_{2n}$ as being distinct from $\beta_j,$ for $j<2n.$ Since $h_{2n}(x_n)\not=0,$ the function
$$
	\lambda(x) = \frac{\beta_{2n}-f_{2n-1}(x)}{H(x)h_{2n}(x)}
$$
is defined and continuous in a neighbourhood $I$ of $x_n,$ with $\lambda(x_n)=0;$ by choosing $I$ smaller we can also have that $|\lambda(x)|<\eta$ for $x\in I.$ By the density of $A$ there is some $\alpha_{2n}\in A\setminus \{\alpha_1,\ldots,\alpha_{2n-1}\}$ in $I.$ Write $\lambda_{2n}=\lambda(\alpha_{2n}).$ Then we have that $|\lambda_{2n}|<\eta$ and $f_{2n}(\alpha_{2n})=f_{2n-1}(\alpha_{2n})+H(\alpha_{2n})\lambda_{2n}h_{2n}(\alpha_{2n})=\beta_{2n}.$ We have chosen $\alpha_{2n}$ and $\lambda_{2n}.$ 

The choice of $\alpha_{2n+1}$ is easy.  
We choose the first of the $a_j $  different from 
$\alpha_1, \ldots, \alpha_{2n} $
and call it $\alpha_{2n+1}. $ 

Since $\Phi(\alpha_{2n+1})h_{2n+1}(\alpha_{2n+1})\not=0,$ the linear function 
$$\beta(\lambda) =  f_{2n}(\alpha_{2n+1})+\Phi(\alpha_{2n+1})\lambda h_{2n+1}(\alpha_{2n+1})$$ is non-constant. Hence $J_n=\{\beta(\lambda):|\lambda|<\epsilon\}$ is a non-empty open interval and, since $B$ is dense, we may choose an element of $(B\cap J_n)\setminus \{\beta_1,\cdots,\beta_{2n}\},$ which we call $\beta_{2n+1}.$ The element $\beta_{2n+1}$ by definition has the form $\beta_{2n+1}=\beta(\lambda)$ for a  certain $\lambda,$ with $|\lambda|<\epsilon.$ We denote this $\lambda$ by $\lambda_{2n+1}.$
If $\epsilon$ is sufficiently small, then $\lambda_{2n+1}$ satisfies  the (finitely many) smallness conditions we have imposed on the $\lambda_j's.$
With this choice of $\beta_{2n+1}$ and $\lambda_{2n+1},$ we have $f_{2n+1}(\alpha_{2n+1})=\beta_{2n+1}.$ 

For $\beta_{2(n+1)}$ we choose the first of the $b_j$'s different from 
$\beta_1, \ldots,\beta_{2n+1}.$ 

The induction is complete and it follows that $f$ restricts to a bijection of $A$ onto $B.$ Since $f^\prime(x)>0, x\in \R,$ this is an order isomorphism. 

There only remains to show that the $\lambda_n$'s can be chosen so small that $f$ is universal. 
From Claim H, we may choose $\lambda_j$ such that $|H(z)\lambda_jh_j(z)|<2^{-j}/|z|,$ for $z \in E.$ It follows that $|\Phi(z)-f(z)|\le 1/|z|, $ for $z\in E.$ In particular, 
$$
	\left(\Phi(z)-f(z)\right)\longrightarrow 0, \quad \mbox{as} \quad z\rightarrow \infty,  \quad z\in E. 
$$
To see that the universality of $\Phi$ entails that of $f,$ fix an entire function $g,$ a compact set $K$ and a positive number $\epsilon.$ The construction of $\Phi$ assures us that there are arbitrarily large $n,$ such that $K\subset \overline D_n, \, E_n^+=\overline D_n+a_n^+$ and $|\Phi(z+a_n^+)-g(z)|<\epsilon$ on $\overline D_n.$ We may choose such an $n$ so large that $|f(w)-\Phi(w)|<\epsilon$ on $E_n^+.$ This is the same as $|f(z+a_n^+)-\Phi(z+a_n^+)|<\epsilon$ on $\overline D_n.$ Hence, $|f(z+a_n^+)-g(z)|<2\epsilon$ on $\overline D_n$ and {\em a fortiori} on $K.$ We have verified that $f$ is indeed a universal function.  This completes the proof of Theorem \ref{2}.

\end{proof}


\section{Universality and linear dynamics}

In 1929, G. D. Birkhoff \cite{Bi1929} established - by an argument anticipating those of the present paper - the existence of universal entire functions. It turned out that universality is generic. 
That is, ``most" entire functions are universal. More precisely, the family of universal entire functions is residual (it is of Baire category II and its complement is of  category I)  in the space of all entire functions. 
However, the situation for order isomorphisms between countable dense subsets of the reals is quite the opposite. 
 Let $\mathcal E$ denote the space of entire functions; let $\mathcal E_R$  be the ``real" entire functions, that is, the entire functions that map reals to reals; and   let $\mathcal E_\rightarrow$ be the space of functions in $\mathcal E_R,$ whose restrictions to the reals are non-decreasing. 
Then, $\mathcal E_R$ is a closed nowhere dense subset of $\mathcal E$ 
and  $\mathcal E_\rightarrow$ is a closed nowhere dense subset of $\mathcal E_R.$ Thus, the class of universal entire functions constructed here is, within the space of
all entire functions, “topologically thin,” i.e., of first Baire category, whereas the
space of all universal entire functions is ``thick”, i.e., of second category. In other words,
the “hard analysis” driving our
constructions cannot be replaced by “soft” methods. 

Although most entire functions are universal, no explicit example is known. The only known function that has a universality property in the sense of Birkhoff (universality of translations) is the Riemann zeta-function! 
It is not entire, but as close to entire as possible. It has only one pole and that pole is simple. More precisely, the spectacular universality theorem of Voronin \cite{V1975} states that vertical translates of $\zeta(z)$ approach ``frequently"  all functions holomorphic in the strip $1/2<\Re z<1$ having no zeros. Moreover, Bagchi \cite{B1987} has shown that the Riemann Hypothesis is equivalent to the possibility of approximating the function  $\zeta(z)$ itself in this fashion by its own translates (a sort of almost periodicity). Bagchi establishes this formulation of the Riemann Hypothesis in the language of topological dynamics. 

Universality in the sense of the present paper is also connected to 
the burgeoning field of {\em linear} dynamics into which the concept of
universality has evolved within the past thirty or so years. Linear dynamics is a fusion of the (usually nonlinear) study of dynamical systems with the theory of
linear operators on topological vector spaces. In this setting, Birkhoff’s universality theorem becomes the
statement that translation operators on the space of entire functions are hypercyclic. An operator on a linear
topological space is called cyclic if there is a vector whose orbit under the operator’s iterates has dense linear
span. Hypercyclic means that the orbit itself is dense.

Recently the subject of linear dynamics has attained enough maturity to justify two recent books
written by accomplished young researchers (\cite{BM2009}, \cite{G-E2011}). It is possible that the results and
methods of the present paper, given their intrinsically “non-soft” nature,
might be of interest to researchers in this burgeoning area.


\end{document}